\documentclass[12pt]{amsart}
\usepackage{amssymb,amsfonts}
\usepackage{amsmath,amsthm,amsxtra}

\newtheorem{theorem}{Theorem}
\newtheorem{lemma}[theorem]{Lemma}

\numberwithin{equation}{section} 
\numberwithin{theorem}{section} 

\newcommand{\la}{\lambda}
\newcommand{\DD}{\mathbb{D}}
\newcommand{\mfa}{\mathfrak{a}}
\newcommand{\mfg}{\mathfrak{g}}
\newcommand{\mfk}{\mathfrak{k}}
\newcommand{\mfp}{\mathfrak{p}}
\newcommand{\al}{\alpha}

\title[The bounded spherical functions]{The bounded spherical functions on the Cartan motion group}
\author{Sigurdur Helgason}
\begin{document}
\thanks{2010 MCS \ Primary 43A85, 22E46, 43A90; \ Secondary 22F30}

\begin{abstract}
The bounded spherical functions are determined for a complex Cartan motion group. 
\end{abstract}

\maketitle




\section{Introduction}
Consider a symmetric space $ X = G/K $ of noncompact type, $ G $ being a connected noncompact semisimple Lie group with finite center and $ K $ a maximal compact subgroup. Let $ \mfg = \mfk + \mfp $ be the corresponding Cartan decomposition, $ \mfp $ being the orthocomplement of $ \mfk $ relative to the Killing form $ B (= \langle \ , \ \rangle ) $ of $ \mfg $. Let $ \mfa \subset \mfp  $ be a maximal abelian subspace, $ \Sigma $ the set of root of $ \mfg  $ relative to $ \mfa $, $ \mfa^+ $ a fixed Weyl chamber and $ \Sigma^+ $ the set of roots $ \alpha $ positive on $ \mfa^+. $ Let $ \rho $ denote the half sum of the $ \alpha \in \Sigma^+ $ with multiplicity. The spherical functions on $ X $ (and $ G $) are by definition the $ K $-invariant joint eigenfunctions of the elements in $ \DD (X), $ the algebra of $ G $-invariant differential operators on $ X. $ By Harish-Chandra's result \cite{HC58} the spherical functions on $ X $ are given by 

\begin{equation}
\label{eq0}
\phi_{\la} (gK) = \int_{K} e^{(i \la - \rho) (H (gK))} \ dk, \quad \phi (eK) = 1, 
\end{equation}

\noindent where $ \exp H(g) $ is the $ A $ factor in the Iwasawa decomposition $ G = KAN $ ($ N $ nilpotent) and $ \la $ ranges over the space $ \mfa^{\ast}_c $ of complex-valued linear functions on $ \mfa. $ Also, $ \phi_{\la} \equiv \phi_{\mu} $ if and only if the elements $ \la, \mu \in \mfa^{\ast}_c $ are conjugate under $ W. $

Let $ L^{\natural}(G)  $ denote the (commutative) Banach algebra of $ K $-bi-invariant integrable functions on $ G. $ The maximal ideal space of $ L^{\natural}(G) $ is known to consist of the kernels of the spherical transforms  
\[ f \rightarrow \int_{G} f(g) \phi_{-\la} (g) \, dg  \]
 for which $ \phi_{-\la} $ is bounded. These bounded spherical functions were in \cite{HJ69} found to be those $ \phi_{\la} $ for which $ \la $ belongs to the tube $ a^{\ast} + i C (\rho) $ where $ C(\rho) $ is the convex hull of the points $s \rho (s \in W). $
 
 This result is crucial in proving that the horocycle Radon transform is injective on $ L^1(X) $ (\cite{H70}, Ch. II).
 
 \section{The boundedness criterion.}
 In this note we deal with the analogous question for the Cartan motion group $ G_0. $ This group is defined as the semidirect product of $ K $ and $ \mfp  $ with respect to the adjoint action of $ K $ on $ 
\mfp. $ The $ X_0 = G_0 / K $ is naturally identified with the Euclidean space $ \mfp. $ The element $ g_0 = (k,Y) $ actions on $ \mfp $ by
\[  g_0 (Y^{\prime}) = A d(k) Y^{\prime} + Y \quad k \in K, \ Y, Y^{\prime} \in \mfp, \]
\noindent so the algebra $ \DD (X_0) $ of $ G_0-$invariant differential operators on $ X_0 $ is identified with the algebra of $ Ad(K)-$invariant constant coefficient differential operators on $ \mfp. $ The corresponding spherical functions on $ X_0 $ are given by
\begin{equation}
\psi_{\la} (Y) = \int_{K} \ e^{i \la (Ad (k)Y)} \, dk \quad \la \in \mfa^{\ast}_c,
\end{equation}

\noindent and $ \psi_{\la} \equiv \psi_{\mu} $ if and only if $ \la $ and $ \mu $ are $ W-$conjugate. See e.g. \cite{H84}, IV \S 4. Again, the maximal ideal space of $ L^{\natural} (G_0) $ is up to $ W- $invariance identified with the set of $ \la$ in $ \mfa^{\ast}_c $  for which $ \psi_{\la}  $ is bounded. Since $ \rho $ is related to the curvature of $ G/K $ it is natural to expect the bounded $ \psi_{\la} $ to come from replacing $ C(\rho) $ by the origin, in other words $ \psi_{\la} $ is would be expected to be bounded if and only if $ \la $ is real, that is $ \la \in \mfa^{\ast}. $

The bounded criterion in \cite{HJ69} for $ X $ relies on Harish-Chandra's expansion for $ \phi_{\la}, $ combined with the reduction to the boundary components of $ X. $ These are certain subsymmetric spaces of $ X. $These tools are not available for $ X_0 $ so the ``tangent space analysis'' in \cite{H80} relies on approximating $ \psi_{\la}  $ by $ \phi_{\la} $ suitably modified. Although several papers (\cite{BC86}, \cite{R88}, \cite{SO05} ) are directed to asymptotic properties of the function $ \psi_{\la} $ the boundedness question does not seem to be addressed there. In this note we only give a partial solution through the following result. 

\begin{theorem}
Assume the group $ G $ complex. The spherical function $ \psi_{\la} $ on $ G_0 $ is bounded if and only if $ \la $ is real, i.e. $ \la \in \mfa^{\ast}. $
\end{theorem}

For $ \la \in \mfa^{\ast}_c $ let $ \la = \xi + i \eta $ with $ \xi, \eta \in \mfa^{\ast}. $ It remains to prove that if $ \la_0 = \xi_0 + i \eta_0 $ with $ \eta_0 \neq 0 $ then $ \psi_{\la_0} $ is unbounded. For $ \la \in \mfa^{\ast}_c $ let $ A_{\la} \in \mfa_c $ be determined by $ \langle A_{\la}, H \rangle = \la (H) \ (H \in \mfa)$. With $ i \la_0 = i \xi_0 - \eta_0 $ we may by the $ W $-invariance of $ \psi_{\la} $ in $ \la $ assume that $ -A_{\eta_0} \in \overline{\mfa^+} $ (the closure of $ \mfa^+. $)

Let $ U \subset W $ be the subgroup fixing $ \la_0 $ and $ V \subset W $ the subgroup fixing $ \eta_0. $ Then $ U \subset V $ and

\begin{equation}
\label{eq1}
\psi_{s \xi_0 + i \eta_0} = \psi_{\xi_0 + i \eta_0 } \quad \mbox{ for } s \in V.
\end{equation}

In addition we assume that for the lexicographic ordering of $ \mfa^{\ast} $ defined by the simple roots $ \alpha_1, \ldots, \alpha_{\ell} $ we have $ \xi_0 \geq s \xi_0 $ for $ s \in V. $

In particular, 

\begin{equation}
\label{eq1.5}
\alpha (A_{\xi_0}) \geq 0 \quad \mbox{for } \alpha \in \Sigma^+ \mbox{ satisfying } \alpha (A_{\eta_0}) = 0.
\end{equation}

\begin{lemma}
\label{lemma2}
The subgroup $ U $ of $ W $ fixing $ \la_0 $ is generated by the reflections $ s_{\alpha_i} $ where $ \alpha_i $ is a simple root vanishing at $ A_{\la_0}. $
\end{lemma}

\begin{proof}
We first prove that some of the $ \al_i  $ vanishes at $ A_{\la_0}. $ The group $ U $ is generated by the $ s_{\al} $ for which $ \al > 0 $ vanishes on $ \la_0 $ (\cite{H78}, VII, Theorem 2.15). If $ \al $ is such then $ \al (-A_{\eta_0}) = 0 $ and since $ \al = \Sigma_j n_j \al_j \ (n_j \neq 0 \mbox{ in } \mathbb{Z}^+)$ and $ \al_j (-A_{\eta_0} )\geq 0 $ we see that each of these $ \al_j $ vanishes on $ A_{-\eta_0}. $ Since $ \al (A_{\xi_0}) = 0 $ and $ \al_j (A_{\xi_0}) \geq 0 $ by (\ref{eq1.5}) for each $ j $ we deduce $ \al_j (A_{\xi_0}) =0. $

Let $ U^{\prime} $ denote the subgroup $ U $ generated by those $ s_{\al_i} $ with $ \al_i $ vanishing at $ \la_0. $ For each $ \al > 0 $ mentioned above we shall prove $ \al = s \al_p $ where $ s \in U^{\prime} $ and $ \al_p $ is simple and vanishes at $ A_{\la_0}. $ We shall prove this by induction on $ \sum_{i} m_i $ if $ \al = \sum m_i \al_i $ ($ m_i \neq 0$ in $ \mathbb{Z}^+ $). The statement is clear if $ \sum m_i = 1 $ so assume $ \sum m_i > 1. $ Since $ \langle \al, \al \rangle > 0$ we have $ \langle \al, \al_k \rangle > 0$ for some $ k $ among the indices $ i $ above. Then $ \al \neq \al_k $ (by $ \sum m_i > 1 $). Since $ s_{\al_k} $ permutes the positive roots $ \neq \al_k $ we have $ s_{\al_k} \al \in \sum^+ $ and $ s_{\al_k} \al = \sum_j m^{\prime}_j \al_j (m^{\prime}_j \in \mathbb{Z}^+)$ and by the choice of $ k, \sum m^{\prime}_j < \sum m_i. $ Now $ \al (A_{i \la_0}) =0 $ and $ \al_i (A_{-\eta_0}) \geq 0 $ so for each $ i  $ in the sum for $ \al $ above, $ \al_i (A_{\eta_0}) = 0. $ Hence by (\ref{eq1.5}) $ \al_i (A_{\xi_0}) = 0. $ In particular $ s_{\al_k} \in U. $ Thus the induction assumption applies to $ s_{\al_k} \al $ giving a $ s^{\prime} \in U^{\prime} $ for which $ s_{\al_k} \al = s^{\prime} \al_p. $ Hence $ \al = s \al_p $ with $ s \in U^{\prime}.  $ But then $ s_{\al} = s s_{\al_k} s^{-1} $ proving the lemma.
\end{proof}

Using Harish-Chandra's integral formula \cite{HC57} Theorem 2 we have

\begin{equation}
\label{eq2}
\psi_{\la} (\mathrm{exp} H) = c_0 \frac{\sum_{s \in W} \epsilon (s) e^{i \langle sA_{\la}, H \rangle}}{\pi(H) \, \pi (A_{\la}) } \qquad \langle H \in \mfa \rangle, 
\end{equation}

\noindent where $ c_0 $ is a constant, $ \langle \ , \ \rangle $  the Killing form, $ \epsilon (s) = \mathrm{det} \, s $ and $ \pi $ the product of the positive roots. If $ \eta_0 $ is regular so $ -A_{ \eta_0} \in \mfa^+ $ then $ V = U = \{ e \} $ and $ \pi (A_{\la_0}) \neq 0. $ Fix $ H_0 \in \mfa^+ $ and $ \la = \la_0 $ in the sum (\ref{eq2}). With $ H = t H_0 (t > 0) $ the term in (\ref{eq2}) with $ s = e $ will outweigh all the others as $ t \rightarrow + \infty $ so $ \psi_{\la} $ is unbounded. 

We now consider the case $ \pi (A_{\la_0}) = 0. $ 

Let $ \pi^{\prime} $ denote the product of the positive roots $ \beta_1, \ldots, \beta_r $ vanishing at $ \la_0 $ and $ \pi^{\prime \prime} $ the product of the remaining positive roots. For $ \la = \la_0 $ we want to divide the factor $\pi^{\prime} (\la_0)  $ into the numerator of (\ref{eq2}). We do this by multiplying (\ref{eq2}) by $\pi^{\prime} (\la),  $ then applying the differential operator $ \partial (\pi^{\prime}) $ in the variable $ \la $ and finally setting $ \la = \la_0. $ The theorem then follows from the following lemma. 

\begin{lemma}
\label{lemma3}
Let $ \eta_0 \neq 0. $ Then the function
\[   \zeta_{\la} (H) = \frac{\sum_{s \in W} \epsilon (s) e^{i \langle sA_{\la}, H \rangle}}{\pi(A_{\la}) } \]
\noindent is for the case $ \la = \la_0 $ unbounded on $ \mfa^+. $
\end{lemma}

\begin{proof}
We have 
\[ \pi^{\prime} (\la) \zeta_{\la} (H) = \frac{1}{\pi^{\prime \prime} (\la)} \sum_{s \in W} \epsilon(s) e^{i \langle sA_{\la}, H \rangle}. \]
Applying $ \partial (\pi^{\prime}) = \partial (\beta_1) \ldots \partial (\beta_r)  $ in $ \la  $ and putting $ \la = \la_0 $ we see that

\begin{equation}
\label{eq3}
c \ \zeta_{\la_0} (H) = \sum_{s \in W} P_s (H) e^{i \langle sA_{\la_0}, H \rangle}.
\end{equation}

Here $ c  $ is a constant and $ P_s $ the polynomial
\[ P_s(H) = \left[ \partial (\pi^{\prime})_{\la} \left( \epsilon (s) \frac{1}{\pi^{\prime \prime} (\la)} e^{i s \la (H)}  \right) \right]_{\la = \la_0} \, e^{-i s \la_0 (H)}\]
\noindent whose highest degree term is a constant times 

\begin{equation}
\label{eq4}
\epsilon(s)\frac{1}{\pi^{\prime \prime} (\la_0)} (s \pi^{\prime}) (H).
\end{equation}

We do not need the exact value of $ c $ but for $ r = 2,3, $ respectively, it equals (with $ x_{ij} = \langle \alpha_i, \alpha_j \rangle $)

\[ x^2_{12} + x_{11}^{} x_{22}^{}, \quad x_{11}^{} x_{23}^2 + x_{22}^{}x_{13}^2 + x_{33}^{}x_{12}^2 + x_{11}^{}x_{22}^{}x_{33}^{} + 2 x_{12}^{}x_{13}^{}x_{23}^{}. \]

We break the sum (\ref{eq3}) into two parts, sum over $ V $ and sum over $ W \backslash V. $ For the first we consider $ \Sigma_V $ as $ \Sigma_{V/U} \, \Sigma_U. $ Then (\ref{eq3}) can be written

\begin{equation}
\label{eq5}
c \ \zeta_{\la_0} (H) = e^{- \eta_0 (H)}   \bigg[ \sum_{V/U} e^{is \xi_0 (H)} \sum_{\sigma \in U} P_{s \sigma} (H) \bigg] + \sum_{W \backslash V} P_s (H) e^{is \la_0 (H)} .
\end{equation}

We put here $ H^{\prime} = -A_{\eta_0}, $ let $ H_0 \in \mfa^+ $ be arbitrary and set $ H = tH_0 (t > 0). $ Then the second term in (\ref{eq5}) equals

\begin{equation}
\label{eq6}
\sum_{s \notin V} P_s (tH_0) e^{is \xi_0 (tH_0)} \, e^{\langle sH^{\prime}, tH_0 \rangle}.
\end{equation}
By a standard property of $ \mfa^+  $ we have 

\[ \langle H_1, H_2 \rangle \geq \langle sH_1, H_2 \rangle \quad \mbox{ if } H_1, H_2 \in \mfa^+ \]
so taking limit, 
\[ \langle sH^{\prime} - H^{\prime}, H \rangle \leq 0, \quad H \in \mfa^+. \]

If $ s \notin V $ then $sH^{\prime} - H^{\prime} \neq 0  $. Thus  the map $ H \rightarrow \langle sH^{\prime} - H^{\prime}, H \rangle  $ is open from $ \mfa $ to $ \mathbb{R} $ mapping $ \mfa^+ $ into $ \{ t \leq 0  \}, $ not taking there the boundary value 0. Hence we get

\begin{equation}
\label{eq7}
\langle H^{\prime}, H_0 \rangle > \langle H^{\prime}, sH_0 \rangle \quad \mbox{ for } s \notin V.
\end{equation}

Equivalently, $ \mathrm{dist} \, (H_0, H^{\prime}) < \mathrm{dist} \, (H_0, sH^{\prime} )$ \ for $ s \notin V. $

Consider $ (\ref{eq5}) $ with $ H = tH_0. $ Assume the expression in the bracket has absolute value with $ \limsup_{t \rightarrow + \infty}  \neq 0. $ Considering (\ref{eq7}) the first term in (\ref{eq5}) would have exponential growth larger than that of each term in (\ref{eq6}). 

Thus $ c \neq 0 $ and 
\[ \lim_{t \rightarrow + \infty} \zeta_{\la_0} (tH_0) = \infty \]
\noindent implying Lemma \ref{lemma3} in this case. 

We shall now exclude the possibility that the quantity in the bracket in (\ref{eq5}) (with $ H = tH_0 $) has absolute value with $ \lim \mathrm{sup}_{t \rightarrow  \infty} = 0. $  For this we use the following elementary result of Harish--Chandra \cite{HC58}, Corollary of Lemma 56: Let $ a_1, \ldots a_n $ be nonzero complex numbers and $ p_0, \ldots p_n $ polynomials with complex coefficients. 

Suppose

\begin{equation}
\label{eq7.5}
\limsup_{t \rightarrow \infty} \bigg{|} p_0 (t) + \sum_{j =1}^{n} p_j (t) e^{a_j t} \bigg{|} \leq a
\end{equation}

\noindent for some $ a \in \mathbb{R}. $ Then $ p_0 $ is a constant and $ |p_0| \leq a. $ This implies the following result.

Let $ k_1 \ldots k_n \in \mathbb{R} $ be different and $ p_1, \ldots, p_n $ polynomials. If 

\begin{equation}
\label{eq8}
\limsup_{t \rightarrow + \infty}  \bigg| \sum_{1}^{n} e^{ik_r t} p_r(t) \bigg| = a < \infty
\end{equation}
\noindent then each $ p_r $ is constant. If $ a = 0 $ then each $ p_r = 0. $ This follows from (\ref{eq7.5}) by writing the above sum as  
\[ e^{ik_rt} \left( p_r (t) + \sum_{j \neq r} \ e^{i(k_j - k_r)t} p_j(t)\right).  \]

Note that in the sum

\begin{equation}
\label{eq9}
\sum_{V/U} e^{is \xi_0 (tH_0)} \, \sum_{\sigma \in U} P_{s \sigma} (tH_0)
\end{equation}

\noindent all the terms $ s \xi_0 $ are different ($ s_1, s_2 \in V $ with $ s_1 \xi_0 = s_2 \xi_0 $ implies $ s^{-1}_2 s_1 \in U $).  Thus we can choose $ H_0 \in \mfa^+ $ such that all $ s \xi_0 (H_0) $ are different. 

We shall now show that one of the polynomial in (\ref{eq9}), namely the one for $ s = e, $

\begin{equation}
\label{2.13} 
\sum_{\sigma \in U} P_{\sigma} (tH_0)
\end{equation}
\noindent is not identically 0. For this note that the highest degree term in $ P_{\sigma} $ is a constant (independent of $ \sigma  $) times

\begin{equation}
\label{2.14}
\epsilon (\sigma) \frac{1}{\pi^{\prime\prime}(\la_0)} (\sigma \pi^{\prime\prime}) (tH_0). 
\end{equation}

Now each $ \sigma $ permutes the roots vanishing at $ A_{\la_0}. $ Hence $ \sigma \pi^{\prime} = \epsilon^{\prime}(\sigma) \pi^{\prime}$ where $ \sigma\rightarrow \epsilon^{\prime}(\sigma) $ is a homomorphism of $ U $ into $ \mathbb{R} $. We now use Lemma \ref{lemma2}. Since each $ s_{\al_i} \in U $ maps $ \al_i $ into $ -\al_i $ and permutes the other positive roots vanishing at $ \la_0 $ we see that $ \epsilon^{\prime} (s_{\al_i}) = -1 = \epsilon (s_{\al_i}). $ Thus by Lemma \ref{lemma2} $ \epsilon^{\prime} (\sigma) = \epsilon (\sigma) $ for each $ \sigma \in U. $ Thus (\ref{2.14}) reduces to

\[ \frac{1}{\pi^{\prime\prime}}\pi^{\prime} (tH_0). \]

This shows that the polynomial in (\ref{2.13}) is not identically 0. In view of (\ref{eq8}) this shows that the $ \limsup $ discussed is $ \neq 0 $ and Lemma \ref{lemma3} established.

\end{proof}

I thank Mogens Flensted-Jensen and Angela Pasquale for useful discussions.

\end{document}